\documentclass[12pt]{amsart}
\usepackage[shortlabels]{enumitem}
\usepackage{amsmath, amsthm, amsfonts, amssymb, mathrsfs, graphicx, stmaryrd}
\usepackage[all]{xy}
\usepackage[usenames, dvipsnames]{color}
\usepackage[margin=1in]{geometry} 
\usepackage[bookmarks, bookmarksdepth=2, colorlinks=true, linkcolor=blue, citecolor=blue, urlcolor=blue]{hyperref}
\usepackage{eucal}
\usepackage{xcolor}

\newcommand{\bN}{\mathbf{N}}

\newcommand{\bR}{\mathbf{R}}

\newcommand{\sX}{\mathscr{X}}

\numberwithin{equation}{section}
\newtheorem{theorem}[equation]{Theorem}

\newtheorem{proposition}[equation]{Proposition}
\newtheorem{lemma}[equation]{Lemma}
\newtheorem{corollary}[equation]{Corollary}

\theoremstyle{definition}
\newtheorem{rmk}[equation]{Remark}
\newenvironment{remark}[1][]{\begin{rmk}[#1] \pushQED{\qed}}{\popQED \end{rmk}}
\newtheorem{eg}[equation]{Example}
\newenvironment{example}[1][]{\begin{eg}[#1] \pushQED{\qed}}{\popQED \end{eg}}
\newtheorem{defnaux}[equation]{Definition}
\newenvironment{definition}[1][]{\begin{defnaux}[#1]\pushQED{\qed}}{\popQED \end{defnaux}}

\let\ol\overline
\let\ul\underline
\let\defn\emph
\renewcommand{\phi}{\varphi}

\DeclareMathOperator{\Spec}{Spec}
\newcommand{\GL}{\mathbf{GL}}
\DeclareMathOperator{\rrk}{rrk}
\DeclareMathOperator{\str}{str}
\DeclareMathOperator{\Div}{Div}

\newcommand{\DOI}[1]{\href{http://doi.org/#1}{\color{purple}{\tiny\tt DOI:#1}}}
\newcommand{\arxiv}[1]{\href{http://arxiv.org/abs/#1}{{\tiny\tt arXiv:#1}}}

\author{Arthur Bik}
\address{Institute for Advanced Study, 1 Einstein Drive, Princeton NJ 08540, USA, and Max-Planck-Institute for Mathematics in the Sciences, Inselstraße 22, 04103 Leipzig, Germany}
\email{\href{mailto:mabik@ias.edu}{mabik@ias.edu}}
\urladdr{\url{http://arthurbik.nl}}
\thanks{AB is partially supported by Postdoc.Mobility Fellowship P400P2\_199196 from the Swiss National Science Foundation and a grant from the Simons Foundation (816048, LC)}

\author{Alessandro Danelon}
\address{Department of Mathematics and Computer Science, Technische Universiteit Eindhoven, Eindhoven, The Netherlands}
\email{\href{mailto:a.danelon@tue.nl}{a.danelon@tue.nl}}
\urladdr{\url{https://adanelon.win.tue.nl/}}
\thanks{AD is supported by Jan Draisma's Vici grant 639.033.514 from the Netherlands Organisation for scientific \hspace*{1em} research (NWO)\begin{minipage}{.018\textwidth}\includegraphics[width=\textwidth]{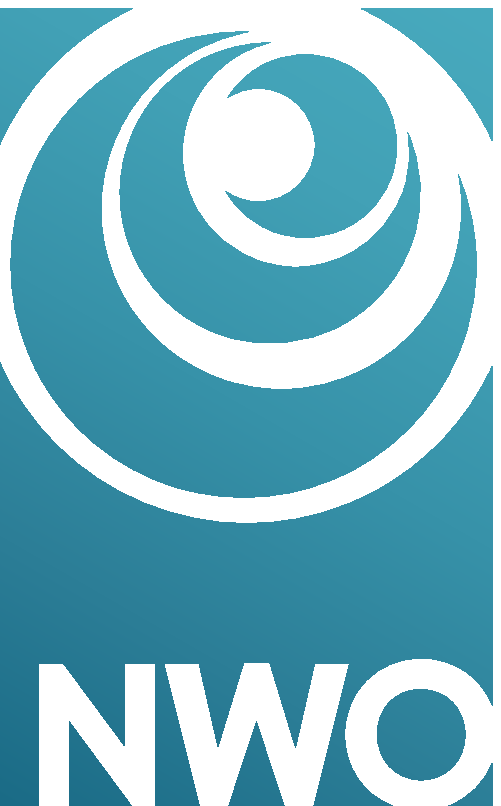}\end{minipage}.}

\author{Andrew Snowden}
\address{Department of Mathematics, University of Michigan, Ann Arbor, MI}
\email{\href{mailto:asnowden@umich.edu}{asnowden@umich.edu}}
\urladdr{\url{http://www-personal.umich.edu/~asnowden/}}

\title{Isogeny classes of cubic spaces}
\date{October 3, 2022}

\begin{document}

\begin{abstract}
A \emph{cubic space} is a vector space equipped with a symmetric trilinear form. Two cubic spaces are \emph{isogeneous} if each embeds into the other. A cubic space is \emph{non-degenerate} if its form cannot be expressed as a finite sum of products of linear and quadratic forms. We classify non-degenerate cubic spaces of countable dimension up to isogeny: the isogeny classes are completely determined by an invariant we call the \emph{residual rank}, which takes values in $\bN \cup \{\infty\}$. In particular, the set of classes is discrete and (under the partial order of embedability) satisfies the descending chain condition.
\end{abstract}

\maketitle
\tableofcontents

\section{Introduction}

For a vector space $V$ over a field $k$, let $P_n(V)=(\Div^n{V})^*$ be the space of degree $n$ polynomial functions on $V$. An \defn{$n$-space} is a vector space $V$ equipped with an element of $P_n(V)$. For example, a 2-space is a quadratic space. For $n>2$, the structure of $n$-forms or $n$-spaces in finite dimensions is quite complicated. However, in the past few years a number of results have been found that reveal more structure in infinite dimensions; see \S \ref{ss:related} for a summary. The purpose of this paper is to establish another result in this direction: we classify non-degenerate cubic spaces (3-spaces) up to isogeny.

\subsection{Background}

Let $V$ and $W$ be $n$-spaces. An \defn{embedding} $W \to V$ is a $k$-linear map such that the form on $V$ restricts to the form on $W$. An \defn{isomorphism} is a bijective embedding. We say that $V$ and $W$ are \defn{isogenous} if each embeds into the other\footnote{We have borrowed this term from \cite{Clark}.}. We define the \defn{strength} of $f \in P_n(V)$, denoted $\str(f)$, to be the infimum of those $s$ for which there is an expression $f = \sum_{i=1}^s g_i h_i$ where $g_i$ and $h_i$ are homogeneous elements of the ring $P(V)=\bigoplus_{j \ge 0} P_j(V)$ of smaller degree than $f$; if there is no such expression for any $s$ then $\str(f)=\infty$. We say that an $n$-space is \defn{non-degenerate} if its defining form has infinite strength.

Let $\sX_n$ be the set of isogeny classes of non-degenerate $n$-spaces of countable dimension. This set is partially ordered by embedability: $[W] \le [V]$ if there exists an embedding $W \to V$. To our knowledge, this  object was first studied in \cite{BDDE}, where the following results were established\footnote{In fact, \cite{BDDE} does not quite work with our notion of isogeny, but one can deduce these statements from the results of \cite{BDDE}.} (for $k$ of algebraically closed of characteristic~0):
\begin{itemize}
\item $\sX_n$ has unique maximal and minimal elements \cite[Theorem~2.9.1]{BDDE}.
\item $\sX_1$ and $\sX_2$ have cardinality one \cite[Example~5.1.2]{BDDE}.
\item $\sX_3$ has at least two elements \cite[Example~2.8.4]{BDDE}.
\end{itemize}
These results motivated us to study $\sX_3$, and this paper is the result of our work.

\subsection{Residual rank}

Let $V$ be a vector space with basis $\{v_i\}_{i \in I}$. We let $x_i \in P_1(V)=V^*$ be the dual vector to $v_i$. We refer to the $x_i$'s as \defn{coordinates} on $V$. Every element of $P_1(V)$ can be uniquely expressed as $\sum_{i \in I} c_i x_i$ where $c_i \in k$ are arbitrary scalars (possibly all non-zero); in other words, the $x_i$'s form a topological basis of the profinite vector space $P_1(V)$. More generally, every element of $P_n(V)$ is a formal $k$-linear combination of degree $n$ monomials in the $x_i$'s.

In what follows, we assume $\operatorname{char}(k) \ne 2,3$. Let $P^{\circ}_2(V)$ be the subspace of $P_2(V)$ consisting of quadratic forms of finite strength, and let $\ol{P}_2(V)$ be the quotient space $P_2(V)/P_2^{\circ}(V)$. The following definition introduces the key concept in this paper:

\begin{definition}
The \defn{residual rank} of $f \in P_3(V)$, denoted $\rrk(f)$, is the dimension of the subspace of $\ol{P}_2(V)$ spanned by the images of the derivatives $\frac{\partial f}{\partial x_i}$ for $i \in I$. The \defn{residual rank} of a cubic space $(V, f)$, denoted $\rrk(V)$, is $\rrk(f)$.
\end{definition}

The derivatives $\frac{\partial f}{\partial x_i}$ appearing above are computed formally: write $f$ as a linear combination of monomials, and then take the derivative of each monomial. The above definition is independent of the choice of basis; in fact, we give a basis-free definition in \S \ref{s:rrk}. We regard residual rank as an element of $\bN \cup \{\infty\}$ (as opposed to a cardinal). Here are two important examples:

\begin{example} \label{ex:rrk0}
Suppose $V$ has coordinates $\{x_i\}_{i \ge 1}$ and $f=\sum_{i=1}^{\infty} x_i^3$. Then $\frac{\partial f}{\partial x_i}=3 x_i^2$ belongs to $P_2^{\circ}(V)$, and thus maps to~0 in $\ol{P}_2(V)$. We therefore have $\rrk(f)=0$.
\end{example}

\begin{example} \label{ex:rrk-inf}
Let $V$ have coordinates $\{x_i\}_{i \ge 1} \cup \{y_{i,j}\}_{i,j \ge 1}$, and let $f=\sum_{i=1}^{\infty} x_i q_i$, where $q_i=\sum_{j=1}^{\infty} y_{i,j}^2$. Then $\frac{\partial f}{\partial x_i}=q_i$, and the $q_i$'s span an infinite dimensional subspace of $\ol{P}_2(V)$. Thus $\rrk(f)=\infty$.
\end{example}

\subsection{Statement of results}

We are finally ready to state our main theorem. In what follows, we assume that $k$ is algebraically closed of characteristic $\ne 2,3$.

\begin{theorem} \label{mainthm}
Let $V$ be a non-degenerate cubic space, and let $W$ be another cubic space of countable (or finite) dimension. Then $W$ embeds into $V$ if and only if $\rrk(W) \le \rrk(V)$.
\end{theorem}

As a corollary, we obtain a precise description of $\sX_3$:

\begin{corollary} \label{maincor}
The map
\begin{displaymath}
\rrk \colon \sX_3 \to \bN \cup \{\infty\}
\end{displaymath}
is a well-defined isomorphism of posets.
\end{corollary}

This corollary clarifies the maximal/minimal result \cite[Theorem~2.9.1]{BDDE}: the minimal class in $\sX_3$ consists of spaces with residual rank zero (such as the one in Example~\ref{ex:rrk0}), while the maximal class consists of spaces with infinite residual rank (such as the one in Example~\ref{ex:rrk-inf}). The corollary also shows that $\sX_3$ is discrete, in the sense that there are not continuous families of isogeny classes, and that $\sX_3$ satisfies the descending chain condition.

\subsection{Higher degrees}
For $n\leq 3$, the set $\sX_n$ is discrete, totally ordered and consists of classes which can be distinguished using invariants that have a simple description. One can wonder whether these statements hold in general. To showcase some of the complexity arising in $\sX_n$ for $n\geq4$ we consider for homogeneous polynomials $p_1,\ldots,p_r\in k[x_1,x_2,\ldots]_{n-2}$ the following two associated objects:
\begin{enumerate}
\item the vector space $\mathrm{span}\{p_1,\ldots,p_r\}$
\item the $n$-space $V_{p_1,\ldots,p_n}$ with coordinates $x_i,y_{j,k},z_{\ell}$ for $i,j,k,\ell\in\mathbb{N}$ and $n$-form
\[
p_1g_1+\ldots+p_rg_r+h,
\]
where $g_j=\sum_{k\in\mathbb{N}}y_{j,k}^2$ and $h=\sum_{\ell\in\mathbb{N}}z_{\ell}^n$.
\end{enumerate}
We say that a finite dimensional subspace $U$ of $k[x_1,x_2,\ldots]_{n-2}$ \textit{specializes} to another subspace $U'$ when $U'=\{p(y_1,y_2,\ldots)\mid p\in U\}$ for some linear $y_1,y_2,\ldots\in k[x_1,x_2,\ldots]$. Two subspaces are \emph{isogeneous} when they both specialize to the other. We have the following result.

\begin{proposition}
Let $p_1,\ldots,p_r,q_1,\ldots,q_s\in k[x_1,x_2,\ldots]_{n-2}$ be homogeneous polynomials. Then $V_{p_1,\ldots,p_r},V_{q_1,\ldots,q_s}$ are isogeneous if and only if $\mathrm{span}\{p_1,\ldots,p_r\},\mathrm{span}\{q_1,\ldots,q_s\}$ are. 
\end{proposition}

For a form $f\in P_n(V)$ and an integer $2\leq k\leq n-2$, we define the {\em $k$-th residual rank} of $f$, denoted $\rrk^{(k)}(f)$, to be the dimension of the subspace of $\ol{P}_{n-k}(V)$ spanned by the images of the $k$-th derivatives of $f$. The \defn{$k$-th residual rank} of an $n$-space $(V, f)$, denoted $\rrk^{(k)}(V)$, is $\rrk^{(k)}(f)$. The proposition shows that several interesting posets can be embedded into $\sX_n$. 

\begin{example}
Let $f\in k[x_1,x_2,\ldots]_2$ be any nonzero polynomial. Then $\rrk(V_f)=0$ and $\rrk^{(2)}(V_f)=1$. However the spans of two polynomials $f,g$ are isogeneous if and only if their ranks (as $2$-forms) coincide. This shows that even when the residual rank is $0$, the second residual rank is not enough to distinguish isogeny classes in $\sX_4$.
\end{example}

\begin{example}
Let $p,q\in k[x_1,x_2,\ldots]_2$ be linearly independent polynomials. Then $\mathrm{span}\{p,q\}$ is isogeneous to a nondegenerate pencil in $k[x_1,\ldots,x_m]_2$ where $m$ is the maximal rank of a form in $\mathrm{span}\{p,q\}$. Such pencils are classified by Segre symbols \cite{fms}. This shows that $\sX_4$ is not totally ordered.
\end{example}

\begin{example}
The poset of isogeny classes of cubic polynomials embeds into $\sX_5$. This shows that $\sX_5$ is not discrete.
\end{example}

\subsection{Why this class of spaces} \label{ss:class}

We defined $\sX_n$ using non-degenerate spaces of countable dimension. We now explain why we use this class of spaces. For finite dimensional spaces, isogeny is equivalent to isomorphism, and there is not much we can say about isomorphism classes of finite dimensional $n$-spaces. Thus we exclude these spaces.

If $V$ is a finite dimensional space then one can add countably many additional basis vectors to $V$ and extend the form by~0. This leads to a degenerate $n$-space of countable dimensional. Understanding isogeny classes of these spaces is as difficult as understanding isomorphisms classes of finite dimensional spaces. We therefore confine our attention to non-degenerate spaces to exclude such examples.

Finally, our arguments rely on enumerating a basis and using inductive arguments, which is why we do not allow uncountable dimension. It is known that quadratic spaces behave very differently in dimensions $\aleph_0$ and $\aleph_1$ \cite[Chapter~2]{Gross}, which suggests that results like Corollary~\ref{maincor} may fail in uncountable dimension.

\subsection{Related work} \label{ss:related}

As we stated, infinite dimensional $n$-spaces, or closely related ideas, have appeared in several recent pieces of work. We now discuss a few. In what follows, we assume $k$ is algebraically closed of characteristic~0 for simplicity.

The notion of strength was introduced by Ananyan and Hochster \cite{AH} in their proof of Stillman's conjecture. They showed that forms of high strength behave approximately like independent variables. Erman, Sam, and Snowden \cite{ess} established a limiting form of this principle: $P(V)$ is (isomorphic to) a polynomial ring, with the ``variables'' being elements of infinite strength. We use results from this circle of ideas in \S \ref{ss:univ}.

Kazhdan and Ziegler \cite{KaZ} showed that strength is universal, in the following sense: if $V$ is a non-degenerate $n$-space then any finite dimensional $n$-space embeds into $V$. Our Theorem~\ref{mainthm} can be seen as a strengthening of this in the $n=3$ case, as it also determines which countable dimensional cubic spaces embed into $V$. The result of Kazhdan and Ziegler was extended by Bik, Draisma, Danelon, and Eggermont \cite{BDDE} to more general tensor spaces. (We note that \cite{KaZ} and \cite{BDDE} do not use the language of $n$-spaces.) This result plays an important role in our proof of Theorem~\ref{mainthm}, and we slightly generalize it in \S \ref{ss:univ}.

Harman and Snowden \cite{homoten} show that there is a universal ultrahomogeneous $n$-space of countable dimension, which is unique up to isomorphism. (In fact, \cite{homoten} works with a more general class of tensor spaces.) This space belongs to the maximal isogeny class.

Let $V$ have dimension $\aleph_0$ and let $X=P_n(V)$, regarded as an infinite dimensional variety on which $\GL(V)$ acts. The equivariant geometry of $X$ has been studied in several recent papers, such as \cite{BDDE,polygeom,Draisma,des}. The most important result in this direction is Draisma's theorem that $X$ is equivariantly noetherian, i.e., the descending chain condition holds for $\GL(V)$-stable closed subsets. For example, the map sending $X\subset P_2(V)$ to the supremal rank of its elements is an isomorphim between the poset of a $\GL(V)$-stable closed subsets of $P_2(V)$ and $\mathbb{N}\cup\{\infty\}$. It would be interesting to understand if there is any connection between this result and our isomorphism between $\sX_3$ and $\mathbb{N}\cup\{\infty\}$.

\subsection{Notation}

We fix an algebraically closed field $k$ throughout. In \S \ref{s:prelim}, we allow any characteristic, but thereafter we exclude characteristics~2 and~3.

\subsection*{Acknowledgments}

We thank Steven Sam for helpful conversations.

\section{Preliminaries} \label{s:prelim}

In this section, $k$ is an algebraically closed field of any characteristic.

\subsection{Embeddings of \textit{n}-spaces}

Let $V$ be a vector space with basis $\{v_i\}_{i \in I}$ and let $x_i$ be dual to $v_i$.
Let $W$ be a vector space with basis $\{w_j\}_{j \in J}$ and let $y_j$ be dual to $w_j$.
Let $\phi \colon W \to V$ be a linear map. Thus $\phi$ assigns to each $w_j$ a corresponding element $\phi(w_j)$ of $V$, which is a finite linear combination of basis vectors. We thus see that $\phi$ can be represented by a $I \times J$ matrix which is column-finite.

The map $\phi$ induces a dual map $\phi^* \colon P_1(V) \to P_1(W)$. Each $\phi^*(x_i)$ can be an infinite linear combination of $y_j$'s, however, each $y_j$ appears with non-zero coefficient in only finitely many $\phi^*(x_i)$'s. In other words, the $J \times I$ matrix representing $\phi^*$ is row-finite. (Of course, this matrix is just the transpose of the matrix for $\phi$.) One can reverse this: a row-finite $J \times I$ matrix defines a linear map $P_1(V) \to P_1(W)$ that is the dual of a unique linear map $W \to V$.

Suppose now that $V$ and $W$ are $n$-spaces, with forms $f(x_{\bullet})$ and $g(y_{\bullet})$. Using the perspective of the previous paragraph, we see that to give an embedding $W \to V$ amounts to giving a linear form $\ell_i \in P_1(W)$ for each $i \in I$ such that each $y_j$ appears in only finitely many $\ell_i$'s and $g(y_{\bullet})=f(\ell_{\bullet})$.
In this case we say that $f$ \textit{specializes} to $g$.
We will use this perspective to define some embeddings when we prove Theorem~\ref{mainthm}.

\subsection{Universality of strength} \label{ss:univ}

Let $V$ be a vector space over $k$. Recall that $P(V)=\bigoplus_{n \ge 0} P_n(V)$. We say that a collection of elements of $P(V)$ has \defn{infinite collective strength} if every non-trivial homogeneous $k$-linear combination has infinite strength. We will require the following theorem, which is a special case of universality of strength (see Remark~\ref{rmk:univ}).

\begin{theorem} \label{thm:univ}
Let $f_1, \ldots, f_r$ be homogeneous elements of $P(V)$ of positive degree and infinite collective strength, and let $c_1, \ldots, c_r \in k$. Then there exists a $v \in V$ with $f_i(v)=c_i$ for all $1 \le i \le r$.
\end{theorem}

\begin{proof}
If $V$ is finite dimensional then the $f_i$'s are linearly independent forms of degree~1, and the result is obvious. Thus assume $V$ is infinite dimensional. By \cite[Theorem~5.5]{ess}, there is a finite dimensional subspace $W$ of $V$ such that, letting $g_i$ be the restriction of $f_i$ to $W$, the forms $g_1, \ldots, g_r$ are a regular sequence in $P(W)$. (Our $P(V)$ is the ring $\bR$ of \cite{ess}, with $A=k$. We note that \cite{ess} works in the case where $V$ has countable dimension, but the argument applies in general.)

Since the $g_i$'s are a homogeneous regular sequence, the $k$-algebra map $k[T_1, \ldots, T_r] \to P(W)$ given by $T_i \mapsto g_i$ is faithfully flat. (This is standard; see \cite[Lemma~3]{AH0}, for instance.) It follows that the map $\Spec(P(W)) \to \Spec(k[T_1, \ldots, T_r])$ of schemes is surjective, and so it is surjective on $k$-points. On $k$-points, this is exactly the map $W \to k^r$ given by $v \mapsto (g_1(v), \ldots, g_r(v))$. Thus the result follows.
\end{proof}

We will mostly use the following variant of the theorem:

\begin{corollary} \label{cor:univ}
Let $V$ be a vector space with:
\begin{itemize}
\item A cubic form $f \in P_3(V)$ of infinite strength.
\item Quadratic forms $q_1, \ldots, q_r \in P_2(V)$ of infinite collective strength.
\item Quadratic forms $q'_1, \ldots, q'_s \in P_2(V)$ each of finite strength.
\item Linear forms $\ell_1, \ldots, \ell_t \in P_1(V)$.
\end{itemize}
Given elements $a, b_1, \ldots, b_r$ in $k$, we can find a $v \in V$ such that
\begin{displaymath}
f(v)=a, \quad q_i(v)=b_i, \quad q'_i(v)=0, \quad \ell_i(v)=0,
\end{displaymath}
for all $i$ that make sense.
\end{corollary}
\begin{proof}
Since $q'_i$ has finite strength, we can write $q'_i=\sum_{j=1}^{n_i} \ell_{i,j}^2$ for some linear forms $\ell_{i,j}$. Let $L \subset P_1(V)$ be the span of the $\ell_i$'s and $\ell_{i,j}$'s, and let $\ell'_1, \ldots, \ell'_a$ be a basis for $L$. The forms
\begin{displaymath}
f, q_1, \ldots, q_r, \ell'_1, \ldots, \ell'_q
\end{displaymath}
have infinite collective strength. Thus by Theorem~\ref{thm:univ}, there is a vector $v \in V$ such that $f(v)=a$, $q_i(v)=b_i$, and $\ell'_i(v)=0$ for all $i$. This implies that $q'_i(v)=0$ and $\ell_i(v)=0$ for all $i$, and so the result follows.
\end{proof}

\begin{remark} \label{rmk:univ}
Let $\ul{d}=(d_1, \ldots, d_r)$ be a tuple of positive integers. A \defn{$\ul{d}$-space} is a vector space $V$ equipped with forms $f_i \in P_{d_i}(V)$ for each $1 \le i \le r$. In this context, ``universality of strength'' means that if $V$ is a $\ul{d}$-space whose defining forms have infinite collective strength then any finite dimensional $\ul{d}$-space $W$ embeds into $V$. For $r=1$, this can be easily deduced from \cite[Theorem~1.9]{KaZ}. For arbitrary $r$ (and, in fact, arbitrary collections of Schur functors), but for $k$ of characteristic~0, it was proved in \cite[Corollary~2.6.3]{BDDE}. Theorem~\ref{thm:univ} establishes it for arbitrary $r$ and $k$, but only when $W$ is 1-dimensional. Our method, relying on commutative algebra, differs from the approaches of \cite{KaZ} and \cite{BDDE}.
\end{remark}

\subsection{Isogeny classes of quadratic spaces}

The following is the analog of Theorem~\ref{mainthm} for quadratic spaces. We will use it in the proof of Theorem~\ref{mainthm}. This result is likely well-known (and follows from \cite[Example~5.1.2]{BDDE} in characteristic~0), but we include a proof to be complete.

\begin{proposition} \label{prop:qmax}
Suppose $\operatorname{char}(k) \ne 2$. Let $V$ be a non-degenerate quadratic space and let $W$ be a quadratic space of dimension $\le \aleph_0$. Then there is an embedding $W \to V$.
\end{proposition}

\begin{proof}
We claim that there exist orthonormal vectors $\{v_i\}_{i \ge 1}$ in $V$. This can be deduced from \cite[\S 2.2]{Gross}, but can also be proved directly as follows. Let $q \in P_2(V)$ be the given quadratic form on $V$, which has infinite strength by assumption, and write $\langle-,- \rangle$ for the associated symmetric bilinear form. Suppose we have found orthonormal vectors $v_1, \ldots, v_{n-1}$. Let $\lambda_i=\langle v_i, - \rangle$, regarded as an element of $P_1(V)$. The $\lambda_i$ are linearly independent, and thus have infinite collective strength. We can therefore find a $v_n$ satisfying $q(v_n)=1$ and $\lambda_i(v_n)=0$ for $1 \le i \le n-1$ by Theorem~\ref{thm:univ}. The claim therefore follows by induction.

Now, let $V_n \subset V$ be the span of $v_1, \ldots, v_n$, and let $V_{\infty}=\bigcup_{n \ge 1} V_n$. Write $W=\bigcup_{n \ge 1} W_n$ with $W_n$ finite dimensional. We claim that there exist embeddings $\phi_n \colon W_n \to V_{\infty}$ that are compatible, in the sense that $\phi_{n+1} \vert_{W_n} = \phi_n$. Of course, this will show that $W$ embeds into $V_{\infty}$, and thus into $V$. Suppose we have constructed $\phi_1, \ldots, \phi_n$. It is not difficult to see that there is some embedding $\phi_{n+1} \colon W_{n+1} \to V_{\infty}$. Let $m$ be such that $\phi_n$ and $\phi_{n+1}$ take values in $V_m$. By Witt's theorem, the embeddings $\phi_n$ and $\phi_{n+1} \vert_{W_n}$ of $W_n$ into $V_m$ differ by an automorphism of $V_m$. Thus modifying $\phi_{n+1}$ by this automorphism, we have $\phi_{n+1} \vert_{W_n} = \phi_n$, as required.
\end{proof}

\begin{corollary}
The set $\sX_2$ is a singleton.
\end{corollary}

\section{Residual rank} \label{s:rrk}

From now on, $k$ is algebraically closed and $\operatorname{char}(k) \ne 2,3$. Let $V$ be a vector space. An element $v \in V$ induces a derivation $\partial_v$ of $P(V)$, given by partial derivative. Explicitly, if $V$ has basis $\{v_i\}_{i \in I}$ and $x_i$ is dual to $v_i$ then $\partial_{v_i}(f)$ is just the partial derivative $\frac{\partial f}{\partial x_i}$. Let $f \in P_3(V)$. For $v \in V$, write $f_v$ for the derivative $\partial_v(f)$, which is an element of $P_2(V)$. For $q \in P_2(V)$, let $\ol{q}$ be its image in $\ol{P}_2(V)$; in particular, $\ol{f}_{\! v} \in \ol{P}_2(V)$. We have a linear map
\begin{displaymath}
\ol{q}_f \colon V \to \ol{P}_2(V), \qquad \ol{q}_f(v) = \ol{f}_{\! v}.
\end{displaymath}
We let $\ol{Q}_f$ be the image of this map. The following proposition gives a basis-free characterization of residual rank:

\begin{proposition}
The residual rank of $f \in P_3(V)$ is the dimension of $\ol{Q}_f$.
\end{proposition}

\begin{proof}
Let $\{v_i\}_{i \in I}$ be a basis of $V$ and let $x_i$ be dual to $v_i$. Since $\ol{q}_f$ is linear, the space $\ol{Q}_f$ is the span of the vectors $\ol{q}_f(v_i) = \frac{\partial f}{\partial x_i}$. The dimension of this space is $\rrk(f)$, by definition.
\end{proof}

The following proposition allows one to easily read of $\ol{Q}_f$ in many cases.

\begin{proposition} \label{prop:rrk-formula}
Let $V$ be a vector space with basis $\{v_i\}_{i \in I}$ and let $x_i \in P_1(V)$ be the dual vector to $v_i$. Let $f \in P_3(V)$, and write $f=\sum_{i \in I} x_i q_i$, where $q_j \in P_2(V)$ and each $x_i$ appears in only finitely many $q_j$'s. Then $\ol{f}_{\! v_i} = \ol{q}_i$, and $\ol{Q}_f$ is the span of the $\ol{q}_i$'s.
\end{proposition}

\begin{proof}
For $i \in I$ we have
\begin{displaymath}
f_{v_i} = q_i + \sum_{j \in I} x_j \frac{\partial q_j}{\partial x_i}
\end{displaymath}
By hypothesis, for fixed $i$ the derivative $\frac{\partial q_j}{\partial x_i}$ is non-zero for only finitely many $j$. Thus the sum on the right is of finite strength, and so maps to~0 in $\ol{P}_2(V)$. This completes the proof.
\end{proof}

Recall that $\str(f)$ denotes the strength of $f \in P(V)$. The following proposition establishes a relationship between strength and residual rank.

\begin{proposition} \label{prop:rrk-str}
Let $f,g \in P_3(V)$.
\begin{enumerate}
\item We have $\rrk(f+g)\leq \rrk(f)+\rrk(g)$.
\item We have $\rrk(f) \le \str(f)$.
\end{enumerate}
\end{proposition}

\begin{proof}
(a)
We have $\frac{\partial(f+g)}{\partial x_i}=\frac{\partial f}{\partial x_i}+\frac{\partial g}{\partial x_i}$ for all $i\in I$ and so $\overline{Q}_{f+g}\subset\overline{Q}_f+\overline{Q}_g$.

(b)
If $f$ has infinite strength there is nothing to prove. Suppose $\str(f)$ is finite. By (a) it suffices to consider the case where $f=\ell q$ with $\ell\in P_1(V)$ and $q\in P_2(V)$. We have
\[
\frac{\partial f}{\partial x_i}=\ell(v_i)q+\ell\frac{\partial q}{\partial x_i}
\]
and so $\ol{Q}_f\subset\mathrm{span}\{\overline{q}\}$
\end{proof}

\begin{proposition} \label{prop:rrk-embed}
Let $V$ be a cubic space and let $W$ be a subspace. Then $\rrk(W) \le \rrk(V)$.
\end{proposition}

\begin{proof}
Let $f \in P_3(V)$ be the given form on $V$, and let $g$ be its restriction to $W$. The inclusion $W \to V$ induces a map $P_2(V) \to P_2(W)$, which clearly preserves finite strength forms, and thus induces a map $\pi \colon \ol{P}_2(V) \to \ol{P}_2(W)$. One easily verifies that for $w \in W$ we have $\pi(\ol{q}_f(w))=\ol{q}_g(w)$. Thus $\ol{Q}_g \subset \pi(\ol{Q}_f)$, and so the result follows.
\end{proof}

Finally, we introduce one additional piece of notation. We let $\langle -, -, - \rangle_f$ be the symmetric trilinear form on $V$ defined by $f \in P_3(V)$, and we let $\langle -, - \rangle_q$ be the symmetric bilinear form defined by $q \in P_2(V)$. For $v \in V$, we have
\begin{displaymath}
\langle -, - \rangle_{f_v} = 3 \langle v, -, - \rangle_f.
\end{displaymath}
This relation characterizes $f_v$ uniquely.

\section{Proof of Theorem~\ref{mainthm}: infinite residual rank} \label{s:rrk-inf}

We now prove Theorem~\ref{mainthm} when $V$ has infinite residual rank. 
In this case the statement reads:

\begin{proposition} \label{prop:main-inf}
Let $V$ be a cubic space of infinite residual rank. If $W$ is any cubic space of countable (or finite) dimension then $W$ embeds into $V$.
\end{proposition}

We note that infinite residual rank implies non-degenerate by Proposition~\ref{prop:rrk-str}.

\subsection{Proposition~\ref{prop:main-inf}: a special case}

As a first step in the proof of the proposition, we verify it holds for the following particular cubic space.

\begin{definition}
The cubic space $V(\infty)$ has coordinates $\{x_i\}_{i \geq 1} \cup \{y_{i,j}\}_{i,j \geq 1}$ and cubic form \begin{align*}
f_{\infty} =& \hspace{1.25em} 3x_1(y_{1,1}^2+y_{1,2}^2+y_{1,3}^2+\cdots) \\
&+ 3x_2(y_{2,1}^2+y_{2,2}^2+y_{2,3}^2+\cdots) \\
&+ 3x_3(y_{3,1}^2+y_{3,2}^2+y_{3,3}^2+\cdots) \\
&+ \cdots
\end{align*}
In more compact notation, $f_{\infty}=3\sum_{i \geq 1}x_i q_i$ with $q_i = \sum_{j \geq 1} y_{i,j}^2$.
\end{definition}

It is clear that $V(\infty)$ has infinite residual rank (see Example~\ref{ex:rrk-inf}). Note that if $\{v_i\}_{i \ge 1} \cup \{w_{i,j}\}_{i,j \ge 1}$ is the basis of $V(\infty)$ dual to the coordinates then we have
\begin{displaymath}
\langle v_i, w_{i,j}, w_{i,j} \rangle_{f_{\infty}} = 1
\end{displaymath}
for all $i,j \ge 1$; up to permutation, the above are the only triplets of basis vectors on which the form is non-vanishing. The following is the result we are after:

\begin{lemma} \label{lem:main-inf-1}
If $W$ is a cubic space of dimension $\le \aleph_0$ then $W$ embeds into $V(\infty)$.
\end{lemma}

\begin{proof}
Since every finite dimensional cubic spaces embeds into one of dimension $\aleph_0$, we can assume $W$ has dimension $\aleph_0$. Let $g \in P_3(W)$ be the given form, let $\{w_i\}_{i \ge 1}$ be coordinates on $W$, and write $g = \sum_{i=1}^{\infty} w_i g_i$ where $g_i \in P_2(W)$ and $w_i$ appears only appears in $g_1, \ldots, g_i$. Thus $w_1 g_1$ is the sum of all monomials in $g$ containing $w_1$; then $w_2 g_2$ is the sum of all remaining monomials containing $w_2$; and so on.

Since the form $\sum_{i \ge 1} z_i^2$ has infinite strength, it follows from Proposition~\ref{prop:qmax} that we can find linear forms $\ell_{i,j} \in P_3(W)$ such that $g_i=\sum_{j \ge 1} \ell_{i,j}^2$ and each $w_n$ appears in only finitely many $\ell_{i,j}$. Moreover, since $w_1, \ldots, w_{i-1}$ do not appear in $g_i$, we can assume they do not appear in $\ell_{i,j}$ for $j \ge 1$. Define a map $P_1(V(\infty)) \to P_1(W)$ by
\begin{displaymath}
x_i \mapsto \tfrac{1}{3} w_i, \qquad y_{i,j} \mapsto \ell_{i,j}
\end{displaymath}
It is clear that under this map $f_{\infty}$ specializes to $g$, and so it yields an embedding $W \to V(\infty)$.
\end{proof}

\begin{remark}
At least in characteristic~0, the lemma follows from  \cite[Lemma 4.4.1]{BDDE}, as the form $f_{\infty}$ is the maximal tensor constructed there.
\end{remark}

\subsection{Proposition~\ref{prop:main-inf}: the general case}

Fix a cubic space $(V, f)$ of infinite residual rank. We must show that any cubic space of dimension $\le \aleph_0$ embeds into $V$. Our strategy is simply to embed the cubic space $V(\infty)$ into $V$. This will establish the proposition by Lemma~\ref{lem:main-inf-1}. In what follows, we write $\ol{q}$ and $\ol{Q}$ in place of $\ol{q}_f$ and $\ol{Q}_f$, as defined in \S \ref{s:rrk}. Thus $\ol{q}(v)$ is simply another notation for $\ol{f}_v$.

Consider a finite table $T$
\begin{center}
\begin{tabular}{l|l}
$v_1$ & $w_{1,1}$, $w_{1,2}$, \ldots, $w_{1,n_1}$ \\
$v_2$ & $w_{2,1}$, $w_{2,2}$, \ldots, $w_{2,n_2}$ \\
\vdots & \vdots \\
$v_r$ & $w_{r,1}$, $w_{r,2}$, \ldots, $w_{r,n_r}$
\end{tabular}
\end{center}
with entries in $V$. We say that $T$ is \emph{good} if the following conditions hold:
\begin{enumerate}
\item The entries are linearly independent.
\item We have $\langle v_i, w_{i,j}, w_{i,j} \rangle_f=1$ for all $1 \le i \le r$ and $1 \le j \le n_i$, and $\langle -,-,- \rangle_f$ vanishes on all other input from the table; precisely, if $x$, $y$, and $z$ are three vectors appearing in the table then $\langle x, y, z \rangle_f \ne 0$ only if (after a permutation) we have $x=v_i$ and $y=z=w_{i,j}$ for some $i$ and $j$.
\item The multiset
\begin{displaymath}
\{ \ol{q}(v_i) \}_{1 \le i \le r} \cup \{ \ol{q}(w_{i,j}) \mid \ol{q}(w_{i,j}) \ne 0 \}_{1 \le i \le r, 1 \le j \le n_i} 
\end{displaymath}
is linearly independent. Explicitly, this means that some $\ol{q}(w_{i,j})$ are allowed to vanish, but after discarding these the $\ol{q}$'s of the remaining table entries are distinct and linearly independent in $\ol{Q}$.
\end{enumerate}
We write $T \subset T'$ to indicate that $T'$ is obtained from $T$ by adding elements. The following two key results show that we can enlarge good tables:

\begin{lemma} \label{lem:main-inf-2}
Given a good table $T$ with $r$ rows and an index $1 \le \ell \le r$, we can find a good table $T \subset T'$ with longer $\ell$th row.
\end{lemma}

\begin{proof}
Without loss of generality, take $\ell=1$. Let $V_0 \subset V$ be the span of the vectors in $T$, and choose any complementary subspace $V_1$ to $V_0$ in $V$. We have a natural map $V \to V/V_1 \cong k^N$, where $N=\dim(V_0)$ is the number of entries in the table. Let $\lambda_1, \ldots, \lambda_N$ be the components of this map. These are linear functionals on $V$, and their joint kernel is $V_1$.

Let $\ol{Q}_0 \subset \ol{Q}$ be the span of the images under $\ol{q}$ of the entries of $T$, and let $\ol{Q}_1$ be any complementary space to $\ol{Q}_0$ in $\ol{Q}$. We have a natural map
\begin{displaymath}
V \stackrel{\ol{q}}{\to} \ol{Q} \to \ol{Q}/\ol{Q}_1 \cong k^M
\end{displaymath}
where $M=\dim(\ol{Q}_0)$. Let $\mu_1, \ldots, \mu_M$ be the components of this map. Thus $x$ belongs to the joint kernel of the $\mu$'s if and only if $\ol{q}(x) \in \ol{Q}_1$.

By Corollary~\ref{cor:univ}, we can find a vector $x \in V$ satisfying the following conditions:
\begin{enumerate}[(1)]
\item $f(x)=0$.
\item $f_{v_1}(x)=3$ (or, equivalently, $\langle v_1, x, x \rangle_f=1$).
\item $f_{v_i}(x)=0$ for $2 \le i \le r$.
\item $f_{w_{i,j}}(x)=0$ for all $w_{i,j}$ in $T$.
\item $\langle a, b, x \rangle_f=0$ for all $a$, $b$ in $T$.
\item $\lambda_i(x)=0$ for all $1 \le i \le N$.
\item $\mu_i(x)=0$ for all $1 \le i \le M$.
\end{enumerate}
We now add $x$ to the first row of $T$ to obtain a new table $T'$. By condition (2), $x$ is non-zero. By condition (6), $x$ belongs to $V_1$. It follows that $x$ is linearly independent from the other table entries, and so $T'$ satisfies (a). It is clear from the above conditions that $T'$ satisfies (b). Condition (7) ensures that $\ol{q}(x)$ belongs to $\ol{Q}_1$, and so either $\ol{q}(x)=0$ or it is linearly independent of the other $\ol{q}$'s in the table. Thus $T'$ satisfies (c).
\end{proof}

\begin{lemma} \label{lem:main-inf-3}
Given a good table $T$, we can find a good table $T \subset T'$ with an additional row.
\end{lemma}

\begin{proof}
Let $V_0$, $V_1$, $\lambda_1, \ldots, \lambda_N$, $\ol{Q}_0$, $\ol{Q}_1$ be as in the previous proof. By assumption the space $\ol{Q}$ has infinite dimension, and so $\ol{Q}_1$ has infinite dimension too. Let $W_1$ be a subspace of $V$ mapping isomorphically to $\ol{Q}_1$ under $\ol{q}$. Thus if $x$ is a non-zero element of $W_1$ then $\ol{q}(x)$ is an element of $\ol{Q}$ that does not belong to $\ol{Q}_0$.

We claim that we can find a non-zero element $x$ of $W_1$ satisfying the following conditions:
\begin{enumerate}[(1)]
\item $f(x)=0$
\item $f_{v_i}(x)$ for all $1 \le i \le r$.
\item $f_{w_{i,j}}(x)=0$ for all $w_{i,j}$ in $T$.
\item $\langle a, b, x \rangle_f=0$ for all $a,b$ in $T$.
\item $\lambda_i(x)=0$ for all $1 \le i \le N$.
\end{enumerate}
Indeed, the above conditions amount to the vanishing of finitely many, say $m$, homogeneous polynomials, and so there is a non-zero solution in $W_1$ since $W_1$ is infinite dimensional (in fact, one can simply restrict to a subspace of $W_1$ of dimension $m+2$).

We now let $T'$ be the table obtained from $T$ by adding a new row with $v_{r+1}=x$ (and with $n_{r+1}=0$, i.e., there are no $w$'s in the new row). Since $x$ is non-zero and belongs to $V_1$ (by condition (5) above), it is linearly independent from the entries of $T$, and so $T'$ satisfies (a). The above conditions clearly imply that $T'$ satisfies (b). Finally, since $x$ is a non-zero element of $W_1$, we find that $\ol{q}(x)$ does not belong to $\ol{Q}_0$, and so $T'$ satisfies (c).
\end{proof}

Proposition~\ref{prop:main-inf} follows from the next lemma:

\begin{lemma} \label{lem:main-inf-4}
The cubic space $V(\infty)$ embeds into $V$.
\end{lemma}

\begin{proof}
Applying the Lemmas~\ref{lem:main-inf-2} and~\ref{lem:main-inf-3} repeatedly, we can construct a chain $T_1 \subset T_2 \subset \cdots$ of good tables such that the number of rows goes to infinity, and the length of any fixed row goes to infinity. Let $T_{\infty}$ be the union of the $T_i$'s. The table $T_{\infty}$ has infinitely many rows (indexed by $\bN$), and each row is infinite (and indexed by $\bN$).

Let $W \subset V$ be the span of the entries of $T_{\infty}$. The entries of $T_{\infty}$ form a basis for $W$. In this basis, the restriction of $f$ is exactly $f_{\infty}$, where $x_i$ and $y_{i,j}$ are dual vector to $v_i$ and $w_{i,j}$. Thus $W$ is isomorphic to $V(\infty)$ as a cubic space, which completes the proof.
\end{proof}

\section{Proof of Theorem~\ref{mainthm}: finite residual rank} \label{s:rrk-fin}

We now prove Theorem~\ref{mainthm} when $V$ has finite residual rank. We recall the statement:

\begin{proposition} \label{prop:main-fin}
Let $V$ be a non-degenerate cubic space of finite residual rank $r$. If $W$ is any cubic space of countable (or finite) dimension and residual rank at most $r$ then $W$ embeds into $V$.
\end{proposition}

We note that the converse statement---if $W$ embeds into $V$ then $\rrk(W) \le r$---follows from Proposition~\ref{prop:rrk-embed}.

\subsection{Proof of Proposition~\ref{prop:main-fin}: a special case}

As in \S \ref{s:rrk-inf}, we first prove the result in a special case. This time, for the following cubic space:

\begin{definition}\label{dfn:finiteAle}
The cubic space $V(r)$ has coordinates
\begin{displaymath}
\{x_i\}_{1 \le i \le r} \cup \{y_{i,j}\}_{i,j \geq 1} \cup \{z_i\}_{i \ge 1}
\end{displaymath}
and cubic form 
\begin{displaymath}
f_r = 3\sum_{i = 1}^r x_i q_i + \sum_{i \geq 1} z_i^3
\end{displaymath}
where $q_i = \sum_{j \geq 1} y_{i,j}^2$.
\end{definition}
It follows from Proposition~\ref{prop:rrk-formula} that $V(r)$ has residual rank $r$. Since the above form specializes to the infinite strength form $\sum_{i \geq 1} z_i^3$ (set each $x_i$ to zero), it has infinite strength.

\begin{lemma} \label{lem:main-fin-2}
Let $W$ be a vector space, and let $a \in P_1(V)$ and $q \in P_2(V)$ with $q$ of finite strength. Then there exists $b_1, \ldots, b_n \in P_1(V)$ such that $a q = \sum_{i=1}^n b_i^3$.
\end{lemma}

\begin{proof}
Put $a=a_1$. Since $q$ has finite strength, we can write $q=\sum_{i=2}^m a_i^2$ for some $a_2, \ldots, a_m \in P_1(V)$. We can thus think of $aq$ as a homogeneous degree-$3$ polynomial in $a_1, \ldots, a_m$. The result now follows since every homogeneous degree-$3$ polynomial can be written as a sum of cubes of linear forms.
\end{proof}

\begin{remark}
In fact, the reasoning in the lemma can be used to prove the following statement: a cubic form is a finite sum of cubes of linear forms if and only if it has finite strength and residual rank zero.
\end{remark}

\begin{lemma} \label{lem:main-fin-1}
Let $W$ be a cubic space of dimension $\le \aleph_0$ and $\rrk(W) \le r$. Then $W$ embeds into $V(r)$.
\end{lemma}

\begin{proof}
As in Lemma~\ref{lem:main-inf-1}, we can assume $\dim(W)=\aleph_0$. Let $g$ be the cubic form on $W$. By assumption, $\dim(\ol{Q}_g) \le r$. Let $\{v_i\}_{i \ge 1}$ be a basis of $W$ such that $\ol{q}_g(v_1), \ldots, \ol{q}_g(v_r)$ span $\ol{Q}_g$ and the $v_i$ with $i>r$ belong to $\ker(\ol{q}_g)$. Let $w_i$ be dual to $v_i$ and write $g=\sum_{i \ge 1} w_i g_i$ where $w_i$ appears only in $g_1, \ldots, g_i$ (see Lemma~\ref{lem:main-inf-1}). By Proposition~\ref{prop:rrk-formula}, we have $\ol{g}_i=\ol{q}_g(v_i)$, and so $g_i$ has finite strength for $i>r$. For $1 \le i \le r$, write $g_i=\sum_{j \ge 1} \ell_{i,j}^2$, where $\ell_{i,j} \in P_1(W)$ and each $w$ coordinate appears in only finitely many $\ell_{i,j}$ (see Lemma~\ref{lem:main-inf-1}).

Now, $g_{r+1}$ has finite strength. Thus by Lemma~\ref{lem:main-fin-2}, we have an expression
\begin{displaymath}
w_{r+1} g_{r+1} = \sum_{j=1}^{n_{r+1}} m_j^3
\end{displaymath}
where $n_{r+1}$ is some integer and the $m_j$ belong to $P_1(W)$. Of course, we can assume that $m_j$ only uses the coordinates $w_i$ with $i \ge r+1$, since these are the only ones appearing in the left side. Similarly, we have an expression
\begin{displaymath}
w_{r+2} g_{r+2} = \sum_{j=n_{r+1}+1}^{n_{r+2}} m_j^3
\end{displaymath}
where now the $m_j$'s only use the $w_i$'s with $i \ge r+2$. Continue in this manner for all remaining $w_i g_i$.

We now define a map $P_1(V(r)) \to P_1(W)$ by
\begin{displaymath}
x_i \mapsto w_i, \qquad y_{i,j} \mapsto \ell_{i,j}, \qquad z_i \mapsto m_i.
\end{displaymath}
It is clear that each $w_i$ appears in only finitely many of the linear forms above. Moreover, it is clear that the form $f_r$ specializes to $g$ under this map, and so we have an embedding $W \to V(r)$.
\end{proof}

\subsection{Proof of Proposition~\ref{prop:main-fin}: the general case}

Fix a non-degenerate cubic space $(V, f)$ of residual rank $r$. We must show that any cubic space of dimension $\le \aleph_0$ and residual rank $\le r$ embeds into $V$. Our strategy is simply to embed the cubic space $V(r)$ into $V$. This will establish the proposition by Lemma~\ref{lem:main-fin-1}. We again write $\ol{q}$ and $\ol{Q}$ for $\ol{q}_f$ and $\ol{Q}_f$.

Let $v_1, \ldots, v_r \in V$ be such that the $\ol{q}(v_i)$ form a basis of $\ol{Q}$. Consider a finite table $T$
\begin{center}
\begin{tabular}{l|l}
$v_1$ & $w_{1,1}$, $w_{1,2}$, \ldots, $w_{1,n_1}$ \\
$v_2$ & $w_{2,1}$, $w_{2,2}$, \ldots, $w_{2,n_2}$ \\
\vdots & \vdots \\
$v_r$ & $w_{r,1}$, $w_{r,2}$, \ldots, $w_{r,n_r}$ \\
\hline
& $u_1$, \ldots, $u_m$
\end{tabular}
\end{center}
with entries in $V$. We say that $T$ is \defn{good} if the following conditions hold:
\begin{enumerate}
\item The entries are linearly independent.
\item We have
\begin{displaymath}
\langle v_i, w_{i,j}, w_{i,j} \rangle_f=1, \qquad \langle u_i, u_i, u_i \rangle_f=1.
\end{displaymath}
Furthermore, if $x$, $y$, and $z$ are entries of $T$ and $\langle x, y, z \rangle_f \ne 0$ then, after applying a permutation, we are in one of three cases: (i) $x=v_i$ and $y=z=w_{i,j}$; (ii) $x=y=z=u_i$; (iii) $x,y,z \in \{v_1, \ldots, v_r\}$.
\item If $a$ is in the right side of the table (i.e., not one of the $v_i$'s) then $f_a$ has finite strength.
\end{enumerate}
(The notion of ``good'' used in \S \ref{s:rrk-inf} will not be used in \S \ref{s:rrk-fin}.)

\begin{lemma} \label{lem:main-fin-2-longer-row}
Given a good table $T$ and an index $1 \le \ell \le r$, we can find a good table $T \subset T'$ with longer $\ell$th row.
\end{lemma}

\begin{proof}
It suffices to treat the case $\ell=1$. Let $V_0$ be the span of the entries of $T$, let $V_1$ be a complementary space to $V_0$, and let $\lambda_1, \ldots, \lambda_N$ be the components of the map $V \to V/V_1 \cong k^N$. Let $\mu_1, \ldots, \mu_r$ be the components of the map $\ol{q} \colon V \to \ol{Q} \cong k^r$.

By Corollary~\ref{cor:univ}, we can find an element $x$ of $V$ satisfying the following conditions:
\begin{enumerate}[(1)]
\item $f(x)=0$.
\item $f_{v_1}(x)=3$.
\item $f_{v_i}(x)=0$ for $2 \le i \le r$.
\item $f_a(x)=0$ for all $a$ on the right side of $T$.
\item $\langle a, b, x \rangle_f=0$ for all $a$ and $b$ in $T$.
\item $\lambda_i(x)=0$ for all $1 \le i \le N$.
\item $\mu_i(x)=0$ for all $1 \le i \le r$.
\end{enumerate}
Let $T'$ be obtained from $T$ by adding $x$ to the first row. We claim that $T'$ is good. By (6), $x$ belongs to $V_1$, and by (2) it is non-zero, and so $x$ is linearly independent from the entries in $T$. Thus $T'$ satisfies (a). It is clear that $T'$ satisfies (b). By (7), $\ol{q}(x)=0$, i.e., $f_x$ has finite strength. Thus $T'$ satisfies (c).
\end{proof}

\begin{lemma} \label{lem:main-fin-3}
Given a good table $T$, we can find a good table $T \subset T'$ with longer final row.
\end{lemma}

\begin{proof}
Use notation as in the previous proof. By Corollary~\ref{cor:univ}, we can find $x \in V$ satisfying the following conditions:
\begin{enumerate}[(1)]
\item $f(x)=1$
\item $f_a(x)=0$ for all $a$ in $T$.
\item $\langle a, b, x \rangle_f=0$ for all $a$ and $b$ in $T$.
\item $\lambda_i(x)=0$ for all $1 \le i \le N$.
\item $\mu_i(x)=0$ for all $1 \le i \le r$.
\end{enumerate}
Let $T'$ be obtained from $T$ by adding $x$ to the final row. We claim that $T'$ is good. Indeed, the element $x$ is linearly independent from the entries of $T$ by (1) and (4). Conditions (1), (2), and (3) guarantee point (b). By (5) we have $\ol{q}(x)=0$, i.e., $f_x$ has finite strength, and so $T'$ satisfies (c).
\end{proof}

The next lemma completes the proof of Proposition~\ref{prop:main-fin}.

\begin{lemma} \label{lem:main-fin-4}
The cubic space $V(r)$ embeds into $V$.
\end{lemma}

\begin{proof}
Let $T_1 \subset T_2 \subset \cdots$ be good tables with each row growing larger, and let $T_{\infty}$ be their union. Let $W$ be the span of the entries of $T_{\infty}$. The entries of $T_{\infty}$ form a basis for $W$. Let $x_1, \ldots, x_r$ be the coordinates of $v_1, \ldots, v_r$, let $y_{i,j}$ be the coordinate of $w_{i,j}$, and let $z_i$ be the coordinate of $u_i$. The restriction of $f$ to $W$ has the form
\begin{displaymath}
g(x_1, \ldots, x_r) + \sum_{i=1}^r x_i q_i + (z_1^3+z_2^3+\cdots)
\end{displaymath}
where $g$ is some cubic and $q_i=\sum_{j \ge 1} y_{i,j}^2$. Write $g=\sum_{i=1}^m \ell_i^3$ where the $\ell_i$ are linear forms in the $x_j$'s. Let $W'$ be the subspace of $W$ given by $\bigcap_{i=1}^m \ker(z_i+\ell_i)$. We obtain a basis for $W'$ by simply removing $u_1, \ldots, u_m$ from the basis for $W$. The restriction of $f$ to $W'$ is given by
\begin{displaymath}
\sum_{i=1}^r x_iq_i + (z_{m+1}^3+z_{m+2}^3+\cdots)
\end{displaymath}
We thus see that $W'$ is isomorphic to $V(r)$, which completes the proof.
\end{proof}

\end{document}